\documentclass[jgt]{degruyter-journal-a}

\renewcommand{\emph}[1]{\textbf{\textit{#1}}}

\newcommand{\C}{\mathbb{C}}
\newcommand{\cC}{\mathcal{C}}
\newcommand{\R}{\mathbb{R}}
\newcommand{\Z}{\mathbb{Z}}
\newcommand{\N}{\mathbb{N}}
\newcommand{\F}{\mathbb{F}}

\renewcommand{\O}{\Omega}
\newcommand{\e}{\epsilon}
\newcommand{\la}{\lambda}
\newcommand{\s}{\sigma}
\renewcommand{\iff}{\Longleftrightarrow}

\DeclareMathOperator{\Tr}{Tr}
\DeclareMathOperator{\vol}{vol}

\newenvironment{en}
{\begin{enumerate}}
{\end{enumerate}}

\theoremstyle{definition}
\newtheorem{de}{Definition}[section]
\theoremstyle{plain}
\newtheorem{thm}[de]{Theorem}
\newtheorem{prop}[de]{Proposition}
\newtheorem{lem}[de]{Lemma}
\newtheorem{cor}[de]{Corollary}
\newtheorem{rem}[de]{Remark}
\newtheorem{exm}[de]{Example}
\newtheorem{conj}[de]{Conjecture}

\title{Ultraproducts of Quasirandom Groups with Small Cosocles}
\headlinetitle{Ultraproducts of Quasirandom Groups}

\lastnameone{Yang}
\firstnameone{Yilong}
\nameshortone{Y.\,~Yang}
\addressone{3170 Sawtelle Blvd Apt 203, Los Angeles, CA 90066}
\countryone{USA}
\emailone{yy26@math.ucla.edu}

\abstract{
A D-quasirandom group is a group without any non-trivial unitary representation of dimension less than D. Given a sequence of groups with increasing quasirandomness, then it is natural to ask if the ultraproduct will end up with no finite dimensional unitary representation at all. This is not true in general, but we answer this question in the affirmative when the groups in question have uniform small cosocles, i.e., their quotient by small kernels are direct products of finite simple groups.\\
Two applications of our results are given, one in triangle patterns inside quasirandom groups and one in self-Bohrifying groups. Our main tools are some variations of the covering number for groups, different kinds of length functions on groups, and the classification of finite simple groups.}



\acknowledgments{I would like to thank Professor Terence Tao for introducing me to this area and for his patient guidance. I would also like to thank Professor Richard Schwartz, Professor Vitaly Bergelson, Professor Emmanuel Breuillard and Professor Nikolay Nikolov for their helpful inputs, and thank Professor L\'{a}szl\'{o} Pyber for his helpful inputs and for pointing me to a number of very useful references.}

\begin{document}

\section{Introduction}

As an indirect consequence of Kassabov, Lubotzky and Nikolov's paper \cite{KLN}, the following theorem about non-abelian finite simple groups is true.

\begin{thm}
\label{thm:KLN}
An ultraproduct of non-abelian finite simple groups is either finite simple, or has no finite dimensional unitary representation other than the trivial one.
\end{thm}

Definitions related to ultraproducts are presented in Section~\ref{sec:DefUltra} for those unfamiliar with them.

In this paper, we shall show that non-abelian finite simple groups are not the only kind of groups exhibiting such a behavior. It turns out that such a behavior has a very close link to the notion of quasirandom groups, defined by Gowers \cite{G}, and the notion of minimally almost periodic groups, defined by von Neumann and Wigner \cite{vNW}. All representations considered in this paper are over $\C$. We shall informally say that a group is quasirandom when the group is $D$-quasirandom for some large $D$.

\begin{de}
For a positive integer $D$, a group $G$ is \emph{$D$-quasirandom} if it has no non-trivial unitary representation of dimension less than $D$.
\end{de}

\begin{de}
An infinite group is \emph{minimally almost periodic} if it has no nontrivial finite dimensional unitary representation.
\end{de}

A group is minimally almost periodic iff it is $D$-quasirandom for all $D$. Then it is natural to wonder whether some sort of limit of increasingly quasirandom groups would give us a minimally almost periodic group. One such limit to consider is the ultraproduct.

The author will prove the existence of classes of groups with similar results to Theorem~\ref{thm:KLN}. The main theorem is the following Theorem~\ref{main1}.

\begin{de}
\label{def:cosocle}
For a group $G$, we define its \emph{cosocle} $Cos(G)$ to be the intersection of all maximal normal subgroups of $G$.
\end{de}

Let $n$ be any positive integer. Let $\cC_n$ be the class of groups that are arbitrary direct products (not necessarily finite) of finite quasisimple groups and finite groups $G$ whose cosocles contain at most $n$ conjugacy classes of $G$.

\begin{thm}
\label{main1}
For any sequence of groups in $\cC_n$ with quasirandom degree going to infinity, their non-principal ultraproducts will be minimally almost periodic.
\end{thm}

Quasirandom groups are first introduced by Gowers to find groups with no large product-free subset. They can be seen as stronger versions of perfect groups.

\begin{exm}[Gowers \cite{G}]
\label{exmqr}
\ \begin{en}
\item A group (not necessarily finite) is $2$-quasirandom iff it is perfect. The reason is that a non-perfect group has a non-trivial abelian quotient, which in turn has a non-trivial homomorphism into $\mathrm{U}_1(\C)$. A perfect group, on the other hand, can only have the trivial homomorphism into the abelian group $\mathrm{U}_1(\C)$.
\item A finite perfect group with no normal subgroup of index less than $n$ is at least $\sqrt{\log n}/2$-quasirandom. In fact, using a form of Jordan's theorem \cite{Co}, a finite perfect group with no normal subgroup of index less than $n$ is at least $c\log n$-quasirandom for some constant $c$.
\item In particular, a non-abelian finite simple group $G$ is at least $c\log n$-quasirandom if it has $n$ elements. 
\item Conversely, any $D$-quasirandom group must have more than $(D-1)^2$ elements.
\item The alternating group $\mathrm{A}_n$ is $(n-1)$-quasirandom for $n>5$, and the special linear group $\mathrm{SL}_2(F_p)$ is $\frac{p-1}{2}$-quasirandom for any prime $p$.
\end{en}
\end{exm}

Morally, ultraproducts preserve all local properties at the scale of elements. In particular, all element-wise identities are preserved. But global properties of a group, like being finite or finitely generated, might be lost after taking ultraproducts. So one may wonder if a non-principal ultraproduct of increasingly quasirandom groups is always minimally almost periodic. In another words, we want to investigate if quasirandomness can be captured by element-wise properties. This turns out to be false. In particular, we have the following counterexample, pointed out by L\'{a}szl\'{o} Pyber.

\begin{exm}
\label{ctexp}
We recall that a group $G$ (not necessarily finite) is $2$-quasirandom iff $G$ is perfect. We claim that there is a sequence of $D_i$-quasirandom groups $(G_i)_{i\in\Z^+}$ with $\lim_{i\to\infty}D_i=\infty$, whose ultraproduct by any non-principal ultrafilter is not even perfect.

Using the construction of Holt and Plesken \cite[Lemma 2.1.10]{HW}, one may construct a finite perfect group $G_{p,n}$ for each prime $p\geq 5$ and positive integer $n$, such that an element of $G_{p,n}$ cannot be written as a product of less than $n$ commutators, and that the only simple quotient of $G_{p,n}$ is $\mathrm{PSL}_2(\F_p)$, the projective special linear group of $2\times 2$ matrices over the field of $p$ elements. Then by Example~\ref{exmqr} (ii), for any $D$, $G_{p,n}$ is $D$-quasirandom for large enough $p$.

Let $G_i$ be $G_{p_i,i}$, where $(p_i)_{i\in\Z^+}$ is a strictly increasing sequence of primes. Then $G_i$ is $D_i$-quasirandom for some $D_i$ with $\lim_{i\to\infty}D_i=\infty$. Let $g_i\in G_i$ be an element which cannot be written as a product of less than $i$ commutators. Then $g=(g_i)_{i\in\N}$ corresponds to an element of the ultraproduct $G=\prod_{i\to\omega}G_i$ by any ultrafilter $\omega$. When $\omega$ is non-principal, clearly $g$ cannot be written as a product of finite number of commutators in $G$. So $g$ is not in the commutator subgroup of $G$, and thus $G$ is not perfect.
\end{exm}

However, a recent paper by Bergelson and Tao \cite{BT} showed the following theorem, which shed some new light on this inquiry:

\begin{thm}[Bergelson and Tao {\cite[Theorem 49 (i)]{BT}}]
The ultraproduct $\prod_{i\to\omega}\mathrm{SL}_2(\F_{p_i})$ by a non-principal ultrafilter $\omega$ is minimally almost periodic.
\end{thm}

Inspired by this, we can make the following definitions:

\begin{de}
\label{def:qup}
A class $\mathcal{F}$ of groups is a \emph{q.u.p. (quasirandom ultraproduct property) class} if for any sequence of groups in $\mathcal{F}$ with quasirandom degree going to infinity, their non-principal ultraproducts will be minimally almost periodic.
\end{de}

\begin{de}
A class $\mathcal{F}$ of groups is a \emph{Q.U.P. class} if there is an unbounded non-decreasing function $f:\Z^+\to\Z^+$ such that any ultraproduct of any sequence of $D$-quasirandom groups in $\mathcal{F}$ is $f(D)$-quasirandom.
\end{de}

\begin{rem}
A Q.U.P class is automatically a q.u.p. class. It is like an effective version of q.u.p. class, where we are able to keep track of the amount of quasirandomness passed down to the ultraproduct.
\end{rem}

In this paper, the proof of Theorem~\ref{main1} in fact shows that the class $\cC_n$ is a Q.U.P. class. And we immediately have the following corollary:

\begin{cor}
The following classes are Q.U.P.
\begin{en}
\item The class $\cC_{QS}$ of finite quasisimple groups.
\item The class $\cC_{SS}$ of finite semisimple groups.
\item The class $\cC_{CS(n)}$ of finite groups with at most $n$ conjugacy classes in their cosocles.
\end{en}
\end{cor}

All Q.U.P. classes must have a uniformly bounded commutator width, i.e., every element can be written as a product of uniformly bounded number of commutators. In view of this, the following conjecture was suggested by L\'{a}szl\'{o} Pyber.

\begin{conj}
For any integer $n$, the class of perfect groups with commutator width $\leq n$ (i.e., every element of these groups can be written as a product of at most $n$ commutators) is Q.U.P.
\end{conj}

So far, we do not know if there is a non-Q.U.P but q.u.p. class of groups.

Some applications of our results have already been found. In a paper in preparation by Bergelson, Robertson and Zorin-Kranich \cite[Theorem 1.12]{BRZ}, it is shown that a sufficiently quasirandom group in a q.u.p. class will have many ``triangles''. As another application, one may also use our method to find many examples of self-Bohrifying groups. Both applications will be explained in Section~\ref{sec:app} of this paper.

Here we shall briefly outline the sections of this article:

\begin{en}
\item A model case of the alternating groups to illustrate the general idea. (Section~\ref{sec:altexm})
\item A group with a nice covering property is very quasirandom. (Section~\ref{sec:local})
\item Covering properties can ignore small cosocles. (Section~\ref{sec:bp})
\item Quasirandom finite quasisimple groups have nice covering properties. (Section~\ref{sec:finsim})
\item Proof of Theorem~\ref{main1}. (Section~\ref{sec:main1})
\item Applications of our results. (Section~\ref{sec:app})
\end{en}


\section{Definitions relating to Ultraproducts}
\label{sec:DefUltra}

\begin{definition}
A \emph{filter} on $\N$ is a collection $\omega$ of subsets of $\N$ such that:
\begin{en}
\item $\varnothing\notin\omega$;
\item If $X\in\omega$ and $X\subseteq Y$, then $Y\in\omega$;
\item If $X,Y\in\omega$, then $X\cap Y\in\omega$.
\end{en}
An \emph{ultrafilter} is a filter that is maximal with respect to the containment order. A \emph{non-principal ultrafilter} is an ultrafilter that contains no finite subset of $\N$.
\end{definition}

\begin{definition}
Given a sequence of groups $(G_i)_{i\in\N}$, let $G$ be their direct product. Given an ultrafilter $\omega$ on $\N$, let $N:=\{g=(g_i)_{i\in\N}\in G:\{ i\in\N:g_i=e\}\in\omega\}$, which is clearly a normal subgroup of $G$. Then we call $G/N$ the \emph{ultraproduct} of the groups $(G_i)_{i\in\N}$ by $\omega$, denoted by $\prod_{i\to\omega}G_i$.
\end{definition}

\begin{rem}
An ultrafilter $\omega$ is principal (i.e., not non-principal) iff we can find an element $n\in\N$ such that for all subsets $A\subseteq\N$, we have $A\in\omega$ iff $n\in A$. In this case, the corresponding ultraproduct of groups $(G_i)_{i\in\N}$ is isomorphic to $G_i$. Therefore, in practice, the useful ultrafilters are usually non-principal.
\end{rem}

The particular choice of the ultrafilter is not that important. As long as we fix a non-principal ultrafilter, then all the discussion for the rest of the paper will be true for the ultraproduct of this ultrafilter.

Ultraproducts have an interesting property, given by \L o\'{s}' Theorem. Given an ultraproduct $G=\prod_{i\to\omega}G_i$ for an ultrafilter $\omega$, any first-order statement $\phi$ in the language of groups is true for $G$ iff it is true for most of the $G_i$, i.e., $\{i\in\N:\phi\text{ is true for }G_i\}\in\omega$. In particular, this implies that behaviors at the scale of elements are preserved. We shall not need \L o\'{s}' Theorem in this paper, but it could be used as an alternative to Proposition~\ref{bp:cover}.


\section{The Class of Alternating Groups}
\label{sec:altexm}

Let $\mathrm{A}_n$ denote the alternating group of rank $n$, and $\mathrm{S}_n$ denote the symmetry group of rank $n$. We shall show that the class of alternating groups is a Q.U.P. class, as a simple illustration of the general idea to attack Theorem~\ref{main1}.

\subsection{Quasirandom Alternating Groups have nice Covering Properties}

\begin{de}
\ \begin{en}
\item For any subsets $A,B$ of a group $G$, we define the product set $AB=\{ab\in G: a\in A,b\in B\}$. And we define $A^n:=\{a_1a_2\ldots a_n:a_1,...,a_n\in A\}$. 
\item An element $g$ of a group $G$ is said to have \emph{covering number} $K$ if its conjugacy class $C(g)$ has $C(g)^K=G$. 
\item Let $m$ be any positive integer or $\infty$. Then an element $g\in G$ has \emph{the covering property $(K,m)$} if $g^i$ has covering number $K$ for all $1\leq i\leq m$. 
\item A group $G$ has \emph{the covering property $(K,m)$} if it has an element with the covering property $(K,m)$.
\end{en}
\end{de}

\begin{rem}
Note that we use $A^n$ to denote the set of elements that can be expressed as products of EXACTLY $n$ elements of $A$. For example, the cyclic group of order $2$ has no covering property at all. The identity is always an even power of the generator, while the generator is always an odd power of itself. There is no uniform choice of $K$ where every element is a product of $K$ conjugates of the generator.
\end{rem}

\begin{de}
An even permutation $\s\in \mathrm{A}_n$ is \emph{exceptional} if its cycles in the cycle decomposition have distinct odd lengths, or equivalently, if its conjugacy class in $\mathrm{A}_n$ is different from its conjugacy class in $\mathrm{S}_n$.
\end{de}

\begin{lem}[Brenner {\cite[Lemma 3.05]{B}}]
\label{lem:Bren}
If an even permutation $\s\in \mathrm{A}_n$ is fixed-point free and non-exceptional, then $\mathrm{A}_n=C(\s)^4$.
\end{lem}

\begin{prop}
\label{alter}
For any $m\in\Z^+$, $\mathrm{A}_n$ has the covering property $(4,m)$ for large enough $n$.
\end{prop}
\begin{proof}
Pick any odd prime $p>m$, and pick another prime $q>p$.

Since $p,q$ are necessarily coprime, for any large enough integer $n$, we can find positive integers $a,b$ such that $n=ap+bq$. Let $\s\in \mathrm{S}_n$ be a permutation composed of $a$ $p$-cycles and $b$ $q$-cycles, where all cycles are disjoint.

Since $p,q$ are odd, $\s$ is an even permutation in $\mathrm{A}_n$. Furthermore, for large enough $n$, $a$ or $b$ can be chosen to be larger than $1$, so $\s$ will be non-exceptional. Since $\s$ is also fixed-point free by construction, Lemma~\ref{lem:Bren} implies that $\mathrm{A}_n=C(\s)^4$.

Now clearly $\s^i$ will also have a cycle decomposition of $a$ $p$-cycles and $b$ $q$-cycles for all $1\leq i\leq p-1$, and this implies that $\mathrm{A}_n=C(\s^i)^4$ for all $1\leq i\leq p-1$. So $\mathrm{A}_n$ has the covering property $(4,p-1)$. Since $p-1\geq m$, $\mathrm{A}_n$ has the covering property $(4,m)$.
\end{proof}

\begin{cor}
For any $m\in\Z^+$, any $D'$-quasirandom alternating group has the covering property $(4,m)$ for large enough $D'$.
\end{cor}

\subsection{Covering Properties passes to Ultraproducts and implies Quasirandomness}

\begin{lem}
Let $G_i$ be a sequence of groups such that all but finitely many of them have the covering property $(K,m)$. Then any ultraproduct of them by a non-principal ultrafilter will have the covering property $(K,m)$.
\end{lem}
\begin{proof}
Since non-principal ultraproducts ignore finitely many exceptions in the sequence $G_i$, WLOG we may assume all $G_i$ have the covering property $(K,m)$.

For each $G_i$, let $g_i$ be the element of $G_i$ with the covering property $(K,m)$. Then I claim that in any ultraproduct of $G_i$, the element represented by the sequence $(g_i)$ would have the covering property $(K,m)$.

Pick any $1\leq j\leq m$. Then any element of $G_i$ is a product of conjugates of $g_i^j$ by $a_{i,1},...,a_{i,K}\in G_i$. As a result, any element of the ultraproduct is a product of conjugates of $(g_i)^j$ by $(a_{i,1}),...,(a_{i,K})$. Here we use a sequence of elements $(a_i)$ to represent an element in the ultraproduct.
\end{proof}

We now state a special case of Proposition~\ref{crit}, proven in Section~\ref{sec:local}.

\begin{lem}
\label{critalt}
There is a function $f:\Z^+\to\Z^+$ such that for any $m,K\in\Z^+$ with $m>f(D)K^{D^2}$, any group $G$ (not necessarily finite) with the covering property $(K,m)$ is $D$-quasirandom. 
\end{lem}

\begin{prop}
The class of alternating groups is a Q.U.P. class.
\end{prop}
\begin{proof}
For any $D\in\Z^+$, find $m>f(D)4^{D^2}$ and find $D'\in\Z$ such that any $D'$-quasirandom alternating group has the covering property $(4,m)$. Let $G$ be an ultraproduct of $D'$-quasirandom alternating groups. Then $G$ will also have the covering property $(4,m)$. Then by Lemma~\ref{critalt}, $G$ is $D$-quasirandom.
\end{proof}


\section{Covering Properties Imply Quasirandomness}
\label{sec:local}

This section is devoted to obtaining some element-scale properties that guarantee the quasirandomness of a group.

\begin{de}
\ \begin{en}
\item An element $g$ of a group $G$ is said to have \emph{symmetric covering number} $K$ if $C(g)^K C(g^{-1})^K=G$. 
\item Let $m$ be a positive integer or $\infty$. Then an element $g\in G$ has \emph{the symmetric covering property $(K,m)$} if $g^i$ has symmetric covering number $K$ for all $1\leq i\leq m$.
\item A group $G$ has \emph{the symmetric covering property} $(K,m)$ if it has an element $g\in G$ with the symmetric covering property $(K,m)$. 
\item A group $G$ has \emph{the (symmetric) covering property} $(K,m)$ mod $N$ for some normal subgroup $N$ if $G/N$ has the (symmetric) covering property $(K,m)$.
\end{en}
\end{de}

\begin{de}
\ \begin{en}
\item A pair of elements $(g_1,g_2)$ of a group $G$ is said to have \emph{symmetric double covering number} $(K_1,K_2)$ if we have $C(g_1)^{K_1} C(g_1^{-1})^{K_1} C(g_2)^{K_2} C(g_2^{-1})^{K_2}=G$.
\item Let $m_1,m_2$ be positive integers or $\infty$. A pair of elements $(g_1,g_2)$ in $G$ has \emph{the symmetric double covering property $[(K_1,m_1),(K_2,m_2)]$} if $(g_1^i,g_2^j)$ has symmetric double covering number $(K_1,K_2)$ for all $1\leq i\leq m_1,1\leq j\leq m_2$.
\item A group $G$ has \emph{the symmetric double covering property} $[(K_1,m_1),(K_2,m_2)]$ if it has a pair of elements $(g_1,g_2)$ in $G$ with the symmetric double covering property $[(K_1,m_1),(K_2,m_2)]$.
\item A group $G$ has \emph{the symmetric double covering property} $[(K_1,m_1),(K_2,m_2)]$ mod $N$ for some normal subgroup $N$ if $G/N$ has the symmetric double covering property $[(K_1,m_1),(K_2,m_2)]$. 
\end{en}
\end{de}

\begin{rem}
\ \begin{en}
\item Suppose $K<K'$. Then an element with covering number $K$ has covering number $K'$. In general, the (symmetric) covering property $(K,m)$ implies the (symmetric) covering property $(K',m')$ when $K'\geq K,m'\leq m$. A similar statement is also true for the symmetric double covering properties.
\item Any symmetric covering property is always weaker than the corresponding non-symmetric covering property.
\item Any group with the symmetric covering property $(K,m)$ has the symmetric double covering property $[(1,\infty),(K,m)]$. This is easily seen by taking $g_1$ to be the identity, and taking $g_2$ to be the element with the symmetric covering property $(K,m)$.
\item In our definition of the symmetric double covering properties, since $C(g_1)$ and $C(g_2)$ are conjugate invariant subsets of $G$, they necessarily commute, i.e., $C(g_1)C(g_2)=C(g_2)C(g_1)$. So the order of $(K_1,m_1)$ and $(K_2,m_2)$ does not matter. 
\item By imitating the definition of the symmetric double covering properties, one can in fact define the symmetric $n$-tuple covering properties for groups. As $n$ grows larger and larger, the corresponding covering properties will become weaker and weaker. Note that most results throughout this paper would still hold by replacing the symmetric double covering properties by the symmetric $n$-tuple covering properties, though for our purpose here, the symmetric double covering properties are enough.
\end{en}
\end{rem}

The proof of Proposition~\ref{crit} will be the main part of this section. Let us first state the proposition and some corollaries.

\begin{prop}[Local criterion for quasirandomness]
\label{crit}
There is a function $f:\Z^+\to\Z^+$ such that, for any $K_1,m_1,K_2,m_2\in\Z^+$ with $m_i>f(D)K_i^{D^2}$ for $i=1,2$, any group $G$ (not necessarily finite) with the symmetric double covering property $[(K_1,m_1),(K_2,m_2)]$ is $D$-quasirandom.
\end{prop}

We shall fix this function $f$ from now on.

\begin{cor}
For any $K,m\in\Z^+$ with $m>f(D)K^{D^2}$, any group $G$ (not necessarily finite) with the symmetric double covering property $(K,m)$ is $D$-quasirandom.
\end{cor}

\begin{cor}
For any $K,m\in\Z^+$ with $m>f(D)K^{D^2}$, any group $G$ (not necessarily finite) with the covering property $(K,m)$ is $D$-quasirandom.
\end{cor}

\begin{rem}
We note here that a partial converse, Corollary~\ref{cor:conv}, of the above result is true. I.e., quasirandomness implies a nice covering property mod cosocle. The proof of this converse will be presented in Section~\ref{sec:main1}.
\end{rem}

We shall first explore some geometric structures of $\mathrm{U}_D(\C)$.

\begin{de}
The \emph{Hilbert-Schmidt norm} of an $n$-by-$n$ complex matrix $A$ is $||A||=\sqrt{\Tr(A^*A)}$.
\end{de}

\begin{lem}
\label{lem:volume}
\ \begin{en}
\item The Lie group $\mathrm{U}_D(\C)$ has a Riemannian metric $d:\mathrm{U}_D(\C)\times \mathrm{U}_D(\C)\to\R$ such that $d(A,B)=||B-A||$ for all $A,B\in \mathrm{U}_D(\C)$. The norm here is the Hilbert-Schmidt norm.
\item This metric is bi-invariant in the sense that $d(AB,AC)=d(BA,CA)=d(B,C)$ for all $A,B,C\in \mathrm{U}_D(\C)$. 
\item This metric induces a Haar measure, and the volume of $\mathrm{U}_D(\C)$ under this Haar measure is finite, and $vol(\mathrm{U}_D(\C))=\frac{(2\pi)^{D(D+1)/2}}{1!2!\ldots (D-1)!}$. We shall denote this constant by $v_D$ from now on.
\item Under the metric $d$, $\mathrm{U}_D(\C)$ has non-negative Ricci curvature everywhere.
\item There is a function $c:\Z^+\to\R^+$, such that a geodesic ball of radius $r$ in $\mathrm{U}_D(\C)$ will have volume bounded by $c(D)r^{D^2}$. We shall fix this function $c$ from now on.
\end{en}
\end{lem}
\begin{proof}
These are very standard facts. See, e.g., \cite{CompLie} and \cite{CurvLie}.
\end{proof}

\begin{de}
Let $G$ be any group. A non-negative function $\ell:G\to\R$ is called a \emph{length function} if it has the following properties.
\begin{en}
\item $\ell(g)=0$ iff $g$ is the identity element.
\item $\ell$ is symmetric, i.e., $\ell(g)=\ell(g^{-1})$ for all $g\in G$.
\item $\ell$ is conjugate invariant, i.e., $\ell(ghg^{-1})=\ell(h)$ for all $g,h\in G$.
\item $\ell$ satisfies the triangle inequality, i.e., $\ell(gh)\leq \ell(g)+\ell(h)$ for all $g,h\in G$.
\end{en}
A \emph{pseudo length function} is a non-negative function $\ell:G\to\R$ satisfying (ii), (iii) and (iv) above.
\end{de}

\begin{lem}
\label{lem:coverlength}
Let $G$ be a group, and suppose $g_1,g_2\in G$ have symmetric double covering number $(K_1,K_2)$. Let $\phi:G\to H$ be any homomorphism and let $\ell$ be a length function of $H$. Then for all $g\in G$, we have $\ell(\phi(g))\leq 2K_1\ell(\phi(g_1))+2K_2\ell(\phi(g_2))$.
\end{lem}
\begin{proof}
For any $g\in G$, $g$ can be written as the product of $K_1$ conjugates of $g_1$, $K_1$ conjugates of $g_1^{-1}$, $K_2$ conjugates of $g_2$ and $K_2$ conjugates of $g_2^{-1}$. So by triangle inequality and the conjugate invariance of $\ell$, we have 
\begin{align*}
\ell(\phi(g))\leq& K_1\ell(\phi(g_1))+K_1\ell(\phi(g_1^{-1}))+K_2\ell(\phi(g_2))+K_2\ell(\phi(g_2^{-1}))\\
\leq& 2K_1\ell(\phi(g_1))+2K_2\ell(\phi(g_2)).
\end{align*}
\end{proof}

\begin{prop}
The function $\ell:\mathrm{U}_D(\C)\to\R$ defined by $\ell(A)=d(A,I)$ is a length function.
\end{prop}
\begin{proof}
Let $A,B$ be any unitary matrices. 

\textit{Positivity:} 
Clearly $\ell(A)=d(A,I)\geq 0$. And we have $$\ell(A)=0 \iff d(A,I)=0 \iff A=I.$$
\textit{Symmetry:}
$$\ell(A)=d(A,I)=d(AA^{-1},IA^{-1})=d(I,A^{-1})=\ell(A^{-1}).$$
\textit{Conjugate Invariance:}
$$\ell(BAB^{-1})=d(BAB^{-1},I)=d(BA,B)=d(A,I)=\ell(A).$$
\textit{Triangle Inequality:}
$$\ell(AB)=d(AB,I)\leq d(AB,B)+d(B,I)=d(A,I)+d(B,I)=\ell(A)+\ell(B).$$
\end{proof}

We shall use $\ell$ to denote this length function from now on.

\begin{lem}
\label{lem:packing}
For any $\e>0$ and any integer $m>\frac{v_D}{c(D)\e^{D^2}}$, any $m$ points in $\mathrm{U}_D(\C)$ will have two points with distance smaller than $\e$. Here $v_D$ and $c(D)$ are as in Lemma~\ref{lem:volume}.
\end{lem}
\begin{proof}
This follows from a volume packing argument.

Since our metric is bi-invariant, each ball of radius $\frac{\e}{2}$ in $\mathrm{U}_D(\C)$ has the same volume $\vol(B_{\e/2})$. So by our assumption on $m$, we have 
$$m>\frac{v_D}{c(D)\e^{D^2}}\geq \frac{\vol(\mathrm{U}_D(\C))}{\vol(B_{\e/2})}.$$

Now for any $m$ points in $\mathrm{U}_D(\C)$, suppose any two of them have distance larger than $\e$. Then the balls of radius $\frac{\e}{2}$ centered at these $m$ points will be disjoint and contained in $\mathrm{U}_D(\C)$, which is impossible. So two of the points have distance smaller than $\e$.
\end{proof}

\begin{lem}
\label{lem:mustexp}
Any non-trivial cyclic subgroup of $\mathrm{U}_D(\C)$ contains an element of length larger than $\sqrt{2}$.
\end{lem}
\begin{proof}
Let $A$ be any nontrivial element of $\mathrm{U}_D(\C)$ of finite order. Let $\la_1,...,\la_D$ be its eigenvalues, and WLOG say $\la_1\neq 1$. Then $\la_1$ is a primitive $n$-th root of unity for some $n$. Replacing $A$ by a proper power of itself, we may assume that $\la_1$ is an $n$-th root of unity closest to $-1$. Then in particular, $|\la_1-1|>\sqrt{2}$.

Then we know 
$$\ell(A)^2=\Tr(A-I)^*(A-I)=\sum_{i=1}^D |\la_i-1|^2\geq |\la_1-1|^2>2.$$

Now suppose $A$ has infinite order. Let $\la_1,...,\la_D$ be its eigenvalues, and WLOG say $\la_1\neq 1$. Then $\la_1$ is an element of infinite order on the unit circle. Replacing $A$ by a proper power of itself, we may assume that $\la_1$ is arbitrarily close to $-1$. Then in particular, $|\la_1-1|>\sqrt{2}$. Then we are done by the same computation.
\end{proof}

\begin{proof}[Proof of Proposition~\ref{crit}]
For any $\e_1,\e_2>0$, pick $m_1>\frac{v_D}{c(D)\e_1^{D^2}}$ and $m_2>\frac{v_D}{c(D)\e_2^{D^2}}$. For any unitary representation $\phi:G\to \mathrm{U}_D(\C)$ of a group $G$ with the symmetric double covering property $[(K_1,m_1),(K_2,m_2)]$, we may find elements $g_1,g_2\in G$ for this symmetric double covering property.

Now consider the points $I,\phi(g_1),\phi(g_1^2),...,\phi(g_1^{m_1})$. By Lemma~\ref{lem:packing}, since $m_1>\frac{v_D}{c(D)\e_1^{D^2}}$, we can find two points with distance less than $\e_1$. Say $d(\phi(g_1^s),\phi(g_1^t))<\e_1$ for some $1\leq s<t\leq m_1$. Then 
$$\ell(\phi(g_1^{t-s}))=d(\phi(g_1^{t-s}),I)=d(\phi(g_1^t),\phi(g_1^s))<\e_1.$$ 
So we have $\ell(\phi(g_1^i)) <\e_1$ for some $1\leq i\leq m_1$. Similarly we have $\ell(\phi(g_2^j)) <\e_2$ for some $1\leq j\leq m_2$.

To sum up, there are elements $g_1^i,g_2^j\in G$ with symmetric double covering number $(K_1,K_2)$, and $\ell(\phi(g_1^i))<\e_1$, $\ell(\phi(g_2^j))<\e_2$. So by Lemma~\ref{lem:coverlength}, all elements of $\phi(G)$ would have length smaller than $2K_1\e_1+2K_2\e_2$.

Now pick $\e_1,\e_2$ small enough so that $2K_1\e_1+2K_2\e_2\leq\sqrt{2}$. (Say $\e_1\leq\frac{\sqrt{2}}{4K_1}$ and $\e_2\leq\frac{\sqrt{2}}{4K_2}$.) Then all elements of $\phi(G)$ would have length at most $\sqrt{2}$. But by Lemma~\ref{lem:mustexp}, this means $\phi(G)$ is trivial.

Therefore, a group with the symmetric double covering property $[(K_1,m_1),(K_2,m_2)]$ will be $D$-quasirandom if $m_1\geq f(D) K_1^{D^2}$ and $m_2\geq f(D) K_2^{D^2}$, where $f(D)=\frac{v_D}{c(D)}(2\sqrt{2})^{D^2}$.
\end{proof}

\begin{rem}
Note that the above argument proves Proposition~\ref{crit} for all groups, not necessarily finite. However, if one only needs to prove Proposition~\ref{crit} for finite groups, and only for the covering property $(K,m)$, then a group is $D$-quasirandom if $\frac{m}{K}\gg$ the length ratio of the longest and the shortest closed geodesics of $\mathrm{U}_n(\C)$. So one can interpret the optimal value of $\frac{m}{K}$ as a measure of the ``shape'' of the finite group. The smaller this optimal value is, the ``more rounded'' the finite group looks like.
\end{rem}


\section{Covering Properties and the Cosocle}
\label{sec:bp}

In this section, we will show that a certain nice covering property mod cosocle is equivalent to a weaker covering property of the whole group.

\begin{lem}
\label{lem:repC}
Let $G$ be a group, and let $N$ be a normal subgroup of $G$ contained in its cosocle. Let $C$ be a conjugate invariant symmetric subset of $G$, such that $CN=G$. Then for any non-empty conjugate invariant subset $S\subseteq G$, $SC=S$ iff $S=G$.
\end{lem}
\begin{proof}
Suppose $SC=S$ and $S\neq G$. Then we have $SC^i=S$ for any positive integer $i$. So $S$ must contain the subgroup generated by $C$. Since $C$ is conjugate invariant, the subgroup generated by $C$ is a normal subgroup, and it is a proper normal subgroup since it is contained in $S\neq G$. In particular, $C$ is contained in a maximal normal subgroup $M$ of $G$.

But since $N$ is in the cosocle, it is contained in $M$. So 
$$CN\subseteq MN=M \subsetneq G.$$ 
This is a contradiction.
\end{proof}

\begin{prop}
\label{prop:bcc}
Let $G$ be a group with the symmetric double covering property $[(K_1,m_1),(K_2,m_2)]$ mod $N$ for a normal subgroup $N$ contained in the cosocle, and suppose that $N$ contains exactly $n$ conjugacy classes of $G$. Then $G$ has the symmetric double covering property $[((3n-2)K_1,m_1),((3n-2)K_2,m_2)]$.
\end{prop}
\begin{proof}
Find $g_1,g_2\in G$ such that $(g_1N,g_2N)$ has symmetric double covering number $(K_1,K_2)$ in $G/N$. Let $C:=C(g_1)^{K_1} C(g_1^{-1})^{K_1}C(g_2)^{K_2} C(g_2^{-1})^{K_2}$. Then by assumption, $C$ is mapped surjectively onto $G/N$ through the quotient map. So $CN=G$.

Now $N$ contains exactly $n$ conjugacy classes of $G$. I claim that $C^{3t}$ contains at least $t+1$ conjugacy classes of $G$ in $N$, which would imply that $C^{3n-3}\supseteq N$. Then $C^{3n-2}\supseteq CN=G$, finishing our proof.

We proceed by induction. As a convention we define $C^0$ to be $\{e\}$. Then the claim is true when $t=0$.

Now assume the statement is true for some $t<n$. Then $C^{3t}$ contains $t+1$ conjugacy classes of $G$ in $N$. Let them be $C_1,...,C_{t+1}$. Then we have $C^{3t+1}\supseteq C(\bigcup_{i=1}^{t+1} C_i)$. Suppose for contradiction that $C^{3t+2}$ is disjoint from $C(N-\bigcup_{i=1}^{t+1} C_i)$. Then we observe that
$$C(N-\bigcup_{i=1}^{t+1} C_i)\supseteq CN-C(\bigcup_{i=1}^{t+1} C_i)= G-C(\bigcup_{i=1}^{t+1} C_i)\supseteq G-C^{3t+1}.$$
So $C^{3t+2}\subseteq C^{3t+1}$. Then Lemma~\ref{lem:repC} implies that $C^{3t+2}=C^{3t+1}=G$. This contradicts the assumption that $C^{3t+2}$ is disjoint from $C(N-\bigcup_{i=1}^{t+1} C_i)$.

So, $C^{3t+2}$ intersects with $C(N-\bigcup_{i=1}^{t+1} C_i)$. Let $g$ be an element in this intersection. Then $g\in CC_{t+2}$ for some conjugacy class $C_{t+2}$ of $G$ in $N$ disjoint from $C_1,...,C_{t+1}$. Find $h\in C_{t+2}$ such that $g\in Ch$. Then since $C$ is symmetric, we have $h\in Cg\subseteq C^{3t+3}$. So $C^{3t+3}$ intersects with $C_{t+2}$. Since $C^{3t+3}$ is conjugate invariant, we conclude that $C^{3t+3}$ contains $C_{t+2}$.

Finally, since $e\in C$, we see that $C^{3t+3}$ also contains $C_1,C_2,...,C_{t+1}$. So $C^{3t+3}$ contains $t+2$ conjugacy classes of $G$ in $N$.
\end{proof}

\begin{prop}
\label{prop:bccC}
Let $G$ be a group with the symmetric covering property $(K,m)$ mod $N$ for a normal subgroup $N$ contained in the cosocle, and suppose that $N$ contains exactly $n$ conjugacy classes of $G$. Then $G$ has the symmetric covering property $((3n-2)K,m)$.
\end{prop}
\begin{proof}
Same strategy as Proposition~\ref{prop:bcc}.
\end{proof}


\section{Quasirandom Finite Simple Groups have Nice Covering Properties}
\label{sec:finsim}

In this section we shall show that, for finite quasisimple groups, large quasirandomness will imply a nice covering property. We shall first deal with finite simple groups of bounded ranks in Subsection~\ref{sub:bdd}. Then we shall deal with the case of alternating groups in Subsection~\ref{sub:alt}. Finally, we shall deal with finite simple groups of large ranks by embedding alternating groups into them in Subsection~\ref{sub:large}. The classification of finite simple groups is used in this section.

\begin{de}
For a finite quasisimple group $G$, we define its \emph{rank} $r(G)$ as the following:
\begin{en}
\item When the only simple quotient of $G$ is abelian or sporadic, then $r(G)=1$.
\item When the only simple quotient of $G$ is the alternating group $\mathrm{A}_n$, then $r(G)=n$.
\item When the only simple quotient of $G$ is a group of Lie type, then $r(G)$ is the (twisted) rank of that finite simple group as an algebraic group.
\end{en}
\end{de}

\subsection{Finite simple groups of bounded ranks}
\label{sub:bdd}
\begin{lem}[Stolz and Thom {\cite[Proposition 3.8]{ST}}]
\label{lem:st}
There is a function $K:\Z^+\to\Z^+$ such that, in any finite simple group of Lie type of rank $\leq r$, any non-identity element will have covering number $K(r)$.
\end{lem}

We shall fix this function $K(r)$ from now on.

\begin{lem}[Babai, Goodman and Pyber {\cite[Proposition 5.4]{BGP}}]
\label{lem:bpg}
Let $k$ be any positive integer. Then for any finite simple group $G$, if $|G|\geq k^{k^2}$, then $|G|$ has a prime divisor greater than $k$.
\end{lem}

\begin{prop}
Let $G$ be a finite simple group of rank $\leq r$. For any $m<\infty$, $G$ has the covering property $(K(r),m)$ if $G$ is $D$-quasirandom for large enough $D$.
\end{prop}
\begin{proof}
By choosing $D$ to be larger than some absolute constant, a $D$-quasirandom group $G$ cannot be an abelian group, a sporadic group, or an alternating group of rank $\leq r$. So we only need to consider finite simple groups of Lie type.

Recall that any $D$-quasirandom group must have more than $(D-1)^2$ elements. For any $m\in\Z^+$, let $D$ be an integer $>1+\sqrt{m^{m^2}}$. Then all $D$-quasirandom finite simple groups will have order $>m^{m^2}$, and thus have an element $g$ of prime order $p>m$. Then $g^i$ are non-identity for all $1\leq i\leq p-1$. Then Lemma~\ref{lem:st} states that all these elements have covering number $K(r)$. So $G$ has the covering property $(K(r),m)$.
\end{proof}

\begin{cor}
\label{conv:smallrank}
Let $G$ be a finite quasisimple group of rank $\leq r$. For any $m<\infty$, $G$ has the symmetric covering property $(K(r)\max(3r+1,34),m)$ if $G$ is $D$-quasirandom for large enough $D$.
\end{cor}
\begin{proof}
If a quasisimple group is $D$-quasirandom, then the simple group it covers is $D$-quasirandom. Therefore, it is enough to show that, if a finite simple group $G$ has the covering property $(K,m)$, then any perfect central extension $G'$ of it will have the covering property $(K\max(3r+1,34),m)$.

Let $Z$ be the center of $G'$. Then $Z$ will be the cosocle of $G'$, and the Schur multiplier of the simple group $G$ would provide an upper bound for $|Z|$. Since $G$ has a rank at most $r$, by going through the list of finite simple groups, its Schur multiplier has a size at most $\max(3r+1,34)$. So if $G$ has the covering property $(K,m)$, Proposition~\ref{prop:bccC} implies that $G'$ has the symmetric covering property $(K\max(3r+1,34),m)$.
\end{proof}

\subsection{Alternating groups}
\label{sub:alt}
\begin{prop}
\label{prop:altrank}
Let $G$ be a quasisimple group over an alternating group. Then for any $m<\infty$, $G$ has the symmetric covering property $(20,m)$ if $G$ is $D$-quasirandom for large enough $D$.
\end{prop}
\begin{proof}
If $G$ is $D$-quasirandom for some large $D$, then the alternating group it covers must be $\mathrm{A}_n$ for some large $n$. Then Proposition~\ref{alter} implies that $\mathrm{A}_n$ has the covering property $(4,m)$. Now when $n>7$, $\mathrm{A}_n$ will have a Schur multiplier of 2. So $G$ has the covering property $(20,m)$.
\end{proof}

\subsection{Finite simple groups of large ranks}
\label{sub:large}
The goal of this section is to prove the following proposition.

\begin{prop}
\label{conv:largerank}
There is an absolute constant $K_0$, such that for any $m<\infty$, all finite quasisimple groups of ranks $\geq r$ will have the symmetric covering property $(K_0,m)$ for large enough $r$.
\end{prop}

By the classification of finite simple groups, a finite simple group of rank larger than some absolute constant will have to be a classical finite simple group of Lie type or an alternating group. Any classical finite simple group of Lie type is in one of the following four classes:

\begin{en}
\item The projective special linear groups $\mathrm{PSL}_n(\F_q)$. For large enough $n$, $\mathrm{SL}_n(\F_q)$ are their universal perfect central extensions.
\item The projective symplectic groups $\mathrm{PSp}_n(\F_q)$. For large enough $n$, $\mathrm{Sp}_n(\F_q)$ are their universal perfect central extensions.
\item The projective special unitary groups $\mathrm{PSU}_n(\F_q)$. For large enough $n$, $\mathrm{SU}_n(\F_q)$ are their universal perfect central extensions.
\item The projective Omega groups $\mathrm{P\O}^+_{2n}(\F_q)$, $\mathrm{P\O}^-_{2n}(\F_q)$, or $\mathrm{P\O}_{2n+1}(\F_q)$. Here $\mathrm{\O}_n(\F_q)$ are the commutator subgroups of the special orthogonal groups $\mathrm{SO}_n(\F_q)$, and $\mathrm{P\O}_n(\F_q)=\mathrm{\O}_n(\F_q)/Z(\mathrm{\O}_n(\F_q))$. The plus or minus signs indicate different quadratic forms used to obtain the groups in even dimensions. For large enough $n$, $\mathrm{\O}_n(\F_q)$ are the universal perfect central extensions of $\mathrm{P\O}_n(\F_q)$.
\end{en}

The above statements can be found in any standard textbook in classical groups (e.g., See \cite{CGGA}). It is enough to show Proposition~\ref{conv:largerank} for $\mathrm{SL}_n(\F_q)$, $\mathrm{Sp}_n(\F_q)$, $\mathrm{SU}_n(\F_q),$ and $\mathrm{\O}_n(\F_q)$, since they are the universal perfect central extensions of the simple groups they cover, and since the order of the Schur multipliers of these groups are bounded above by a function of $r$.

We start by analyzing a length function for groups of Lie type over finite fields.

\begin{de}
\label{def:Jordan}
Let $g$ be an $n\times n$ matrix over a finite field $F$. Let $m_g:=\sup_{a\in F^\times}\dim(\ker(a-g))$. Then the \emph{Jordan length} of $g$ is $\ell_J(g):=\frac{n-m_g}{n}$
\end{de}

\begin{prop}
Let $G$ be any subgroup of $\mathrm{GL}_n(F)$ for some finite field $F$. The function $\ell_J$ on $G$ is a pseudo length function.
\end{prop}
\begin{proof}
\textit{Non-negativity:} For any $g\in G$, 
$$m_g=\sup_{a\in F^\times}\dim(\ker(a-g))\leq n.$$ 
So $\ell_J(g)=\frac{n-m_g}{n}\geq 0$.\\

\textit{Symmetry:} For any $g\in G$, any $a\in F^\times$, and any vector $v\in F^n$, we have 
$$v\in\ker(a-g)\iff av=gv \iff g^{-1}v=a^{-1}v \iff v\in\ker(a^{-1}-g^{-1}).$$
As a result, 
$$m_g=\sup_{a\in F^\times}\dim(\ker(a-g))=\sup_{a\in F^\times}\dim(\ker(a^{-1}-g^{-1}))=m_{g^{-1}}.$$ 
So $\ell_J(g)=\ell_J(g^{-1})$.\\

\textit{Conjugate-invariance:} For any $g,h\in G$, any $a\in F^\times$, and any vector $v\in F^n$, we have 
$$v\in\ker(a-g)\iff av=gv\iff ahv=(hgh^{-1})hv\iff hv\in\ker(a-hgh^{-1}).$$
As a result,
$$m_g=\sup_{a\in F^\times}\dim(\ker(a-g))=\sup_{a\in F^\times}\dim(\ker(a-hgh^{-1}))=m_{hgh^{-1}}.$$ 
So $\ell_J(g)=\ell_J(hgh^{-1})$.\\

\textit{Triangle inequality:} For any $g,h\in G$, any $a,b\in F^\times$, and any vector $v\in F^n$, we have
$$v\in \ker(a-g)\cap\ker(a-abh^{-1})\implies gv=av=abh^{-1}v\implies v\in\ker(abh^{-1}-g).$$

So we know $\ker(a-g)\cap\ker(a-abh^{-1})\subseteq \ker(abh^{-1}-g)$. As a result, we have
\begin{align*}
m_{gh}\geq&\dim\ker(ab-gh)\\
\geq&\dim\ker(abh^{-1}-g)\\
\geq&\dim(\ker(a-g)\cap\ker(a-abh^{-1}))\\
\geq&\dim(\ker(a-g))+\dim(\ker(a-abh^{-1}))-n\\
\geq&\dim(\ker(a-g))+\dim(\ker(b-h))-n.
\end{align*}
Since this is true for all $a,b\in F^\times$, therefore $m_g+m_h-n\leq m_{gh}$. So $\ell_J(gh)\leq\ell_J(g)+\ell_J(h)$.
\end{proof}

\begin{lem}
\label{lem:jordan sum}
Given an $n_1\times n_1$ matrix $A$ over a finite field $F$, and an $n_2\times n_2$ matrix $B$ over the same finite field, then $\ell_J(A\oplus B)\geq \frac{n_1}{n_1+n_2}\ell_J(A)+\frac{n_2}{n_1+n_2}\ell_J(B)$.
\end{lem}
\begin{proof}
For any $a\in F^\times$, we have the following 
$$\ker(a-A\oplus B)=\ker((a-A)\oplus(a-B))=\ker(a-A)\oplus\ker(a-B).$$
So $\dim\ker(a-A\oplus B)\leq m_A + m_B$. Since this is true for all $a\in F^\times$, therefore $m_{A\oplus B}\leq m_A+m_B$. So we have
\begin{align*}
\ell_J(A\oplus B)=&\frac{n_1+n_2-m_{A\oplus B}}{n_1+n_2}\\
\geq&\frac{n_1+n_2-m_A-m_B}{n_1+n_2}\\
\geq&\frac{n_1-m_A}{n_1+n_2}+\frac{n_2-m_B}{n_1+n_2}\\
\geq&\frac{n_1}{n_1+n_2}\ell_J(A)+\frac{n_2}{n_1+n_2}\ell_J(B).
\end{align*}
\end{proof}

\begin{lem}[Stolz and Thom {\cite[Lemma 3.11]{ST}}]
\label{lem:Jordanexp}
There is an absolute constant $c_0$, such that for any finite classical quasisimple group of Lie type $G$, and for any $g\in G\setminus Z(G)$, where $Z(G)$ is the center of $G$, then $C(g)^K = G$ for all $K\geq\frac{c}{\ell_J(g)}$.
\end{lem}

In short, elements of large Jordan length will automatically have small covering number. 

The next step is to identify subgroups of these quasisimple groups of Lie type isomorphic to the alternating groups. A key step is to treat elements in alternating groups as matrices, namely the permutation matrices. These are the matrices with exactly one entry of value $1$ in each column and in each row, and $0$ in all other entries. Such an $n\times n$ matrix will act on the standard orthonormal basis of an $n$-dimensional vector space by permutation, and thus will provide an embedding of $\mathrm{S}_n$ into $\mathrm{GL}_n(F)$ for any field $F$. Any such matrix is in $\mathrm{A}_n$ iff it has determinant $1$.

\begin{lem}
If $P$ is an $n\times n$ permutation matrix where its cycle decomposition has $k$ cycles, then we have $\ell_J(P)\geq\frac{n-k}{n}$.
\end{lem}
\begin{proof}
By cycle decomposition, after a change of basis in the vector space, $P$ will be a direct sum of many cyclic permutation matrices. By Lemma~\ref{lem:jordan sum}, it's enough to prove the case when $P$ is a single cycle of length $n$, and show that $\ell_J(P)\geq\frac{n-1}{n}$.

Since $P$ is a single cycle of length $n$, its eigenvalues in the algebraic closure of $F$ are precisely all the $n$-th roots of unity, with multiplicity 1 for each root of unity. So $\dim\ker(a-P)\leq 1$ for all $a\in F^\times$. So $\ell_J(P)\geq\frac{n-1}{n}$.
\end{proof}

\begin{prop}
\label{sl}
There is an absolute constant $K_0$ such that, for any $m<\infty$, for any finite quasisimple group of Lie type of $n\times n$ matrices, if it contains $\mathrm{A}_n$ as permutation matrices, then it will have the covering property $(K_0,m)$ for large enough $n$.
\end{prop}
\begin{proof}
Let $K_0>3c_0$ for the absolute constant $c_0$ in Lemma~\ref{lem:Jordanexp}. Then any element $A$ of Jordan length $\geq\frac{1}{3}$ will have covering number $K_0$ in any finite quasisimple group of Lie type. 

Pick any odd prime $p>m$, and pick another prime $q>p$. For any large enough $n$, we have $n=ap+bq$ for some integers $a>1$, $0<b<p+1$. Then find $\s\in \mathrm{A}_n$ made up of exactly $a$ $p$-cycles and $b$ $q$-cycles, where all cycles are disjoint. This element will be fixed-point free and non-exceptional, and it will have at most $a+b\leq\frac{n}{p}+p$ cycles.

For any finite quasisimple group of Lie type of $n\times n$ matrices, suppose it contains $\mathrm{A}_n$ as permutation matrices. Let $P$ be the matrix corresponding to $\s$. Then we have
$$\ell_J(P)\geq\frac{n-\frac{n}{p}-p}{n}=1-\frac{1}{p}-\frac{p}{n}>\frac{1}{3}.$$ 
The last inequality follows because $p\geq 3$ and $n\geq2p+q>3p$.

So this element will have covering number $K_0$ in $G$. It clearly has order $pq$, and all of its powers coprime to $pq$ will also have the same covering number. So $G$ has the covering property $(K_0,p-1)$.
\end{proof}

\begin{cor}
\label{cor:sl}
For any $m<\infty$, all finite special linear groups of rank $r$ for large enough $r$ will have the covering property $(K_0,m)$. Here $K_0$ is the absolute constant in Proposition~\ref{sl}.
\end{cor}

\begin{prop}
\label{spsuom}
There is an absolute constant $K_0$, such that for any $m<\infty$, we have the following:
\begin{en}
\item For any finite quasisimple group of Lie type of $2n\times 2n$ matrices, if it contains $\mathrm{A}_n$ as $\{P\oplus P:P\in \mathrm{A}_n$ is a permutation $n\times n$ matrix$\}$, then it will have the covering property $(K_0,m)$ for large enough $n$.
\item Let $I_1$ be the 1 by 1 identity matrix. Then for any finite quasisimple group of Lie type of $(2n+1)\times (2n+1)$ matrices, if it contains $\mathrm{A}_n$ as $\{P\oplus P\oplus I_1:P\in \mathrm{A}_n$ is a permutation $n\times n$ matrix$\}$, then it will have the covering property $(K_0,m)$ for large enough $n$.
\item Let $I_2$ be the 2 by 2 identity matrix. Then for any finite quasisimple group of Lie type of $(2n+2)\times (2n+2)$ matrices, if it contains $\mathrm{A}_n$ as $\{P\oplus P\oplus I_2:P\in \mathrm{A}_n$ is a permutation $n\times n$ matrix$\}$, then it will have the covering property $(K_0,m)$ for large enough $n$.
\end{en}
\end{prop}
\begin{proof}
The strategy is identical to Proposition~\ref{sl}. Just take $\s\oplus\s$, $\s\oplus\s\oplus I_1$ or $\s\oplus\s\oplus I_2$ instead of $\s$, and use Lemma~\ref{lem:jordan sum}.
\end{proof}

\begin{de}
A vector space $V$ is a \emph{non-degenerate formed space} if it has a non-degenerate quadratic form $Q$ (the orthogonal case), or a non-degenerate alternating bilinear form B (the symplectic case), or a non-degenerate Hermitian form $B$ (the unitary case).
\end{de}

\begin{lem}[Witt's Decomposition Theorem]
Let $V$ be any non-degenerate formed space over a finite field $F$. Then we have an orthogonal decomposition $V=W\oplus(\bigoplus_{i=1}^n H_i)$ where $W$ is anisotropic of dimension at most $2$, and $H_i$ are hyperbolic planes.
\end{lem}
\begin{proof}
These are standard facts in the geometry of classical groups (e.g., See \cite{CGGA}).
\end{proof}

\begin{prop}
For a non-degenerate formed space, the special isometry group, i.e., the group of isometries of determinant 1, contains an alternating group in one of the ways described by Proposition~\ref{spsuom}.
\end{prop}
\begin{proof}
Let $V$ be any finite dimensional non-degenerate formed space over any finite field $F$. Then we have an orthogonal decomposition $V=W\oplus H$ with an anisotropic space $W$ of dimension at most $2$, and an orthogonal sum of hyperbolic planes $H=\bigoplus_{i=1}^n H_i$.

Then let $(v_i,w_i)$ be a hyperbolic pair generating $H_i$ for each $i$. For any $\s\in \mathrm{A}_n$, we can let $\s$ act by permutation on the set $\{v_1,..,v_n,w_1,...,w_n\}$, such that $\s(v_i)=v_{\s(i)}$ and $\s(w_i)=w_{\s(i)}$.

Now clearly $\{v_1,...,v_n,w_1,..,w_n\}$ is a basis of $H$. So the above action of $\s$ induces a linear transformation $P\oplus P$ on $H$, where $P$ is the $n\times n$ permutation matrix for $\s$. And this $P\oplus P$ is clearly an isometry on $H$ by construction. Now taking the direct sum of $P\oplus P$ on $H$ and the identity matrix on $W$, we shall obtain our desired embedding of $\mathrm{A}_n$ into the full isometry group.

Finally, since $P$ is a permutation matrix for an even permutation, it has determinant 1. Therefore the above embedding of $\mathrm{A}_n$ is in the special isometry group.
\end{proof}

\begin{cor}
\label{cor:spsu}
For any $m<\infty$, any finite symplectic or special unitary group of rank $r$ has the covering property $(K_0,m)$ for large enough $r$. $K_0$ is the absolute constant in Proposition~\ref{spsuom}.
\end{cor}

\begin{cor}
\label{cor:om}
For any $m<\infty$, any $\mathrm{\O}_{2n}^+(\F_q)$, $\mathrm{\O}_{2n+1}(\F_q)$ or $\mathrm{\O}_{2n}^-(\F_q)$ has the covering property $(K_0,m)$ for large enough $n$. $K_0$ is the absolute constant in Proposition~\ref{spsuom}.
\end{cor}
\begin{proof}
Embed $\mathrm{A}_n$ in $\mathrm{SO}_{2n}^+(q)$, $\mathrm{SO}_{2n}^-(q)$ and $\mathrm{SO}_{2n+1}(q)$ in the ways described by Proposition~\ref{spsuom}. After taking the commutator subgroup, the groups $\mathrm{\O}_{2n}^+(q)$, $\mathrm{\O}_{2n}^-(q)$ and $\mathrm{\O}_{2n+1}(q)$ will still contain $\mathrm{A}_n$ through this embedding, because $\mathrm{A}_n$ is its own commutator subgroup. So we may apply Proposition~\ref{spsuom} to $\mathrm{\O}_{2n}^+(q)$, $\mathrm{\O}_{2n}^-(q)$ and $\mathrm{\O}_{2n+1}(q)$ and obtain the desired result.
\end{proof}

Proposition~\ref{conv:largerank} is proven by putting Corollary~\ref{cor:sl}, Corollary~\ref{cor:spsu} and Corollary~\ref{cor:om} together.


\section{Proof of Theorem~\ref{main1}}
\label{sec:main1}

The results of Section~\ref{sec:finsim} can be summarized into the following useful lemma.

\begin{lem}
\label{lem:conv}
For any integer $D$ and any constant $c$, we can find integers $D',K_1,K_2,m_1,m_2$ such that all $D'$-quasirandom finite quasisimple groups have the symmetric double covering property $[(K_1,m_1),(K_2,m_2)]$ such that $m_1>cK_1^{D^2}$, $m_2>cK^{D^2}$.
\end{lem}
\begin{proof}
Let $K_1$ be $\max(20,K_0)$ where the absolute constant $K_0$ is as in Proposition~\ref{conv:largerank}. Pick some $m_1>cK_1^{D^2}$. Find large enough $r$ such that, according to Proposition~\ref{conv:largerank} and Proposition~\ref{prop:altrank}, all finite quasisimple groups (including the alternating case) of ranks $\geq r$ will have the symmetric covering property $(K_1,m_1)$.

Set $K_2:=K(r)\max(3r+1,34)$ as in Corollary~\ref{conv:smallrank}, and pick some $m_2>cK_2^{D^2}$. Then for large enough $D'$, all $D'$-quasirandom finite quasisimple groups will have the symmetric covering property $(K_2,m_2)$.

In all cases, a $D'$-quasirandom finite quasisimple group will have the symmetric double covering property $[(K_1,m_1),(K_2,m_2)]$.
\end{proof}

\begin{rem}
In the above proof, one cannot substitute the double covering properties with the covering properties. To have a covering property $(K,m)$, a finite simple group must either have a large enough rank to accommodate the large m, according to Proposition~\ref{conv:largerank}, or it must have a small enough rank to accomodate the small K, according to Proposition~\ref{conv:smallrank}. So there might be a gap between the ``large enough rank'' and the ``small enough rank'', where the finite simple subgroups in the gap would fail to have the covering property $(K,m)$, no matter how quasirandom they are.

In short, the covering properties of finite quasisimple groups are not necessarily uniform. It is uniform when obtained through increasing ranks, and it is uniform when obtained through base fields of increasing sizes. At least with the techniques in this paper, we cannot combine the two uniformity into one. So we must use the double covering properties.
\end{rem}

\begin{prop}
\label{bp:cover}
Let $G$ be a group with the symmetric double covering property for some parameters, and let $(G_i)_{i\in I}$ be an arbitrary family of groups with the symmetric double covering property for some uniform parameters. Then the following are true:
\begin{en}
\item For any normal subgroup $N$, $G$ has the symmetric double covering property for the same parameters mod $N$.
\item Any quotient group of $G$ has the symmetric double covering property for the same parameters.
\item The group $\prod_{i\in I}G_i$ has the symmetric double covering property for the same parameters.
\item As a result of the (ii) and (iii), any ultraproduct $\prod_{i\to\omega}G_i$ has the symmetric double covering property for the same parameters.
\end{en}
\end{prop}
\begin{proof}
(i), (ii) and (iv) are straightforward.

To see (iii), let $g_{i,1},g_{i,2}\in G_i$ be the pairs giving $G_i$ the symmetric double covering property. Then I claim that $(g_{i,1})_{i\in I},(g_{i,2})_{i\in I}\in \prod_{i\in I}G_i$ is the pair giving the desired symmetric double covering property.

For any element $(g_i)_{i\in I}\in\prod_{i\in I}G_i$, then each $g_i$ is in $G_i$. And by its symmetric double covering property, we know
$$G_i=C(g_{i,1})^{K_1} C(g_{i,1}^{-1})^{K_1} C(g_{i,2})^{K_2} C(g_{i,2}^{-1})^{K_2}.$$ 
So we can find $a_{i,j},b_{i,j}\in G_i$ for $i\in I$ and $1\leq j\leq K_1$, and $c_{i,j},d_{i,j}\in G_i$ for $i\in I$ and $1\leq j\leq K_2$, such that 
$$g_i=(\prod_{1\leq j\leq K_1}(a_{i,j}g_{i,1}a_{i,j}^{-1})(b_{i,j}g_{i,1}^{-1}b_{i,j}^{-1}))(\prod_{1\leq j\leq K_2}(c_{i,j}g_{i,2}c_{i,j}^{-1})(d_{i,j}(g_{i,2})^{-1}d_{i,j}^{-1})).$$
Since the above identity is true for all $i\in I$, we have
\begin{align*}
(g_i)_{i\in I}=&(\prod_{1\leq j\leq K_1}((a_{i,j})_{i\in I}(g_{i,1})_{i\in I}(a_{i,j})_{i\in I}^{-1})((b_{i,j})_{i\in I}(g_{i,1})_{i\in I}^{-1}(b_{i,j})_{i\in I}^{-1}))\\
&(\prod_{1\leq j\leq K_2}((c_{i,j})_{i\in I}(g_{i,2})_{i\in I}(c_{i,j})_{i\in I}^{-1})((d_{i,j})_{i\in I}(g_{i,2})_{i\in I}^{-1}(d_{i,j})_{i\in I}^{-1})).
\end{align*}
So we have proven (iii).
\end{proof}

\begin{cor}
Let $\cC_{QS}$ be the class of finite quasisimple groups. Then $\cC_{QS}$ is a Q.U.P. class.
\end{cor}
\begin{proof}
For any integer $D$, and for the constant $c=f(D)$ as in Proposition~\ref{crit}, we can find $D',K_1,K_2,m_1,m_2$ as in Lemma~\ref{lem:conv}.

Let $G_i$ be a sequence of $D'$-quasirandom groups in $\cC_{QS}$. Then $G_i$ all have the symmetric double covering property $[(K_1,m_1),(K_2,m_2)]$. Then any ultraproduct $G=\prod_{i\to\omega}G_i$ will have the symmetric double covering property $[(K_1,m_1),(K_2,m_2)]$ by Proposition~\ref{bp:cover}. Since $m_1>f(D)K_1^{D^2}$, $m_2>f(D)K^{D^2}$, $G$ is $D$-quasirandom by Proposition~\ref{crit}.
\end{proof}

\begin{cor}[Quasirandomness implies a Nice Covering Property mod Cosocle]
\label{cor:conv}
For any integer $D$, and any constant $c$, we can find integers $D',K_1,K_2,m_1,m_2$ such that all finite $D'$-quasirandom groups have the symmetric double covering property $[(K_1,m_1),(K_2,m_2)]$ mod cosocle, with $m_1>cK_1^{D^2}$, $m_2>cK_2^{D^2}$.
\end{cor}
\begin{proof}
Let $D',K_1,K_2,m_1,m_2$ be exactly as in Lemma~\ref{lem:conv}. Let $G$ be any finite $D'$-quasirandom group.

Let $N$ be the cosocle of $G$. Then $G/N$ is a direct product of $D'$-quasirandom finite simple groups. These simple groups all have the symmetric double covering property $[(K_1,m_1),(K_2,m_2)]$. So by Proposition~\ref{bp:cover}, their product $G/N$ will have this same symmetric double covering property.
\end{proof}

\begin{cor}
Let $\cC_{CS(n)}$ be the class of finite groups with at most $n$ conjugacy classes in their cosocles. Then $\cC_{CS(n)}$ is a Q.U.P. class.
\end{cor}
\begin{proof}
Let $c=f(D)(3n-2)^{D^2}$.

For any integer $D$, and for the constant $c$, we can find $D',K_1,K_2,m_1,m_2$ as in Corollary~\ref{cor:conv}. 

Let $G_i$ be a sequence of $D'$-quasirandom groups in $\cC_{CS(n)}$. Then $G_i$ all have the symmetric double covering property $[(K_1,m_1),(K_2,m_2)]$ mod cosocles. Since the cosocles contain at most $n$ conjugacy classes, by Proposition~\ref{prop:bcc}, $G_i$ all have the symmetric double covering property $[((3n-2)K_1,m_1),((3n-2)K_2,m_2)]$. Then any ultraproduct $G=\prod_{i\to\omega}G_i$ will have the symmetric double covering property $[((3n-2)K_1,m_1),((3n-2)K_2,m_2)]$ by Proposition~\ref{bp:cover}. 

Since $m_1>f(D)[(3n-2)K_1]^{D^2}$, $m_2>f(D)[(3n-2)K]^{D^2}$, $G$ is $D$-quasirandom by Proposition~\ref{crit}.
\end{proof}

\begin{proof}[Proof of Theorem~\ref{main1}]
For any integer $D$, let $c=f(D)(3n-2)^{D^2}$. We can find $D',K_1,K_2,m_1,m_2$ as in Corollary~\ref{cor:conv} and Lemma~\ref{lem:conv}.

Let $G_i$ be a sequence of $D'$-quasirandom groups in $\cC_n$. Then each $G_i$ is a direct product of $D'$-quasirandom groups in $\cC_{QS}\cup\cC_{CS(n)}$. These factor groups must then have the symmetric double covering property $[((3n-2)K_1,m_1),((3n-2)K_2,m_2)]$. By Proposition~\ref{bp:cover}, $G_i$ must also have this symmetric double covering property $[((3n-2)K_1,m_1),((3n-2)K_2,m_2)]$. Then any ultraproduct $G=\prod_{i\to\omega}G_i$ will have the symmetric double covering property $[((3n-2)K_1,m_1),((3n-2)K_2,m_2)]$ by Proposition~\ref{bp:cover}. 

Since $m_1>f(D)[(3n-2)K_1]^{D^2}$, $m_2>f(D)[(3n-2)K]^{D^2}$, Proposition~\ref{crit} implies that $G$ is $D$-quasirandom.
\end{proof}


\section{Applications}
\label{sec:app}
\subsection{Triangles in a quasirandom group}

A quasirandom group usually contains many patterns. For example, Gowers has shown the following result:

\begin{thm}[Gowers {\cite[Theorem 5.1]{G}}]
\label{Gtheorem}
Pick any $\e_1,\e_2>0,0<\alpha<1$. If $G$ is a $D$-quasirandom group for some large enough $D$, then for any subset $A$ of $G$ such that $|A|\geq\alpha |G|$, there are more than $(1-\e_1)\alpha^2|G|$ elements $x\in G$ such that $|A\cap xA|\geq(1-\e_2)\alpha^2|G|$.
\end{thm}

Morally, if we define an \emph{$x$-pair} to be a set $\{y,xy\}$ for some $y\in G$, then this theorem means that any large enough subset of a quasirandom group $G$ will contain many $x$-pairs for many $x$.

Now given a q.u.p. class, we can obtain minimally almost periodic groups via ultraproducts of sequences of increasingly quasirandom groups. Then by applying ergodic theory on the ultraproduct, more patterns similar to that of Theorem~\ref{Gtheorem} might emerge. It is proven by Bergelson, Robertson and Zorin-Kranich \cite{BRZ} that, for a quasirandom group $G$ in a q.u.p class, any large enough subset of $G\times G$ will contain many $x$-triangle for many $x$. 

\begin{de}
Let $g$ be an element of a group $G$. Then a \emph{$g$-triangle} is the set $\{(x,y),(gx,y),(gx,gy)\}\subseteq G\times G$ for some $x,y\in G$.
\end{de}

\begin{thm}[Bergelson, Robertson and Zorin-Kranich {\cite[Theorem 1.12]{BRZ}}]
Let $G$ be contained in a q.u.p. class. For any $\e>0,0<\alpha<1$, there are integers $D,K$ such that, if $G$ is $D$-quasirandom, then for any subset $A$ of $G\times G$ with $|A|\geq\alpha |G|^2$, the set $T_A=\{g\in G: A$ contains more than $(\alpha^4-\e)|G|^2$ triangles$\}$ can cover $G$ with at most $K$ left translates of itself.
\end{thm}

\subsection{Self-Bohrifying groups}

The application in this section is related to topological groups. We shall treat all groups in previous sections as discrete groups.

\begin{de}
A \emph{Bohr compactification} of a topological group $G$ is a continuous homomorphism $b:G\to bG$ such that any continuous homomorphism from $G$ to a compact group factors uniquely through $b$.
\end{de}

\begin{rem}
\ \begin{en}
\item The Bohr compactification exists for any group by the work of Holm \cite{Holm}. It is obviously unique up to a unique isomorphism. 
\item Clearly, a discrete group is minimally almost periodic iff it has trivial Bohr compactification. Note that for a discrete group, any abstract homomorphism from it to another topological group is automatically continuous.
\end{en}
\end{rem}

\begin{de}
A topological group $G$ is said to be \emph{self-Bohrifying} if its Bohr compactification $bG$ is the same abstract group as $G$, but with a compact topology.
\end{de}

By the results and techniques of this paper, one can find many examples of self-Bohrifying groups. In particular, we have the following theorem.

\begin{thm}
\label{app:Bohr}
Let $n$ be a positive integer. Let $G_i$ be a sequence of increasingly quasirandom groups in $\cC_n$, the class defined as in Theorem~\ref{main1}. Then $\prod_{i\in\N} G_i$ is self-Bohrifying as a discrete group.
\end{thm}

\begin{cor}
Let $G_i$ be a sequence of non-abelian finite simple groups of increasing order. Then $\prod_{i\in\N} G_i$ is self-Bohrifying as a discrete group.
\end{cor}

We will prove Theorem~\ref{app:Bohr} by first showing that $\prod_{i\in\N} G_i / \coprod_{i\in\N} G_i$ is minimally almost periodic, and then using a lemma by Hart and Kunen \cite{HK}.

\begin{de}
Let $G_i$ be a sequence of groups.
\begin{en}
\item Their \emph{sum} is the group $\coprod_{i\in\N} G_i=\{g\in\prod_{i\in\N} G_i:$ only finitely many coordinates of $g$ is nontrivial$\}$.
\item Their \emph{reduced product} is the group $\prod_{i\in\N} G_i / \coprod_{i\in\N} G_i$.
\end{en}
\end{de}

\begin{lem}[Hart and Kunen {\cite[Lemma 3.8]{HK}}]
\label{HKtheorem}
Let $\{G_i\}_{i\in\N}$ be a sequence of finite groups. Then $\prod_{i\in\N} G_i$ is self-Bohrifying if all but finitely many $G_i$ are perfect groups, and $\prod_{i\in\N} G_i/\coprod_{i\in\N} G_i$ has trivial Bohr compactification, i.e., $\prod_{i\in\N} G_i/\coprod_{i\in\N} G_i$ is minimally almost periodic.
\end{lem}

\begin{proof}[Proof of Theorem~\ref{app:Bohr}]
All 2-quasirandom groups are perfect. So it is enough to show that the reduced product of $G_i$ is minimally almost periodic, i.e., it is $D$-quasirandom for all $D$.

For any integer $D$, let $c=f(D)(3n-2)^{D^2}$. We can find $D',K_1,K_2,m_1,m_2$ as in Corollary~\ref{cor:conv} and Lemma~\ref{lem:conv}.

Let $G_i$ be a sequence of increasingly quasirandom groups in $\cC_n$. Then all but finitely many $G_i$ will be $D'$-quasirandom. Since we are interested in the reduced product, which is invariant under the change of finitely many coordinates, we may WLOG assume that all $G_i$ are $D'$-quasirandom.

Since $G_i\in\cC_n$, each $G_i$ is a direct product of $D'$-quasirandom groups in $\cC_{QS}\cup\cC_{CS(n)}$. These factor groups must then have the symmetric double covering property $[((3n-2)K_1,m_1),((3n-2)K_2,m_2)]$. By Proposition~\ref{bp:cover}, $G_i$ must also have this symmetric double covering property $[((3n-2)K_1,m_1),((3n-2)K_2,m_2)]$. 

Now by Proposition~\ref{bp:cover}, covering properties are preserved by arbitrary products and quotients. So $\prod_{i\in\N} G_i$ will have this covering property, and the reduced product $\prod_{i\in\N} G_i/\coprod_{i\in\N} G_i$ will also have this covering property. 

Since $m_1>c[(3n-2)K_1]^{D^2}$, $m_2>c[(3n-2)K]^{D^2}$, the reduced product is $D$-quasirandom by Proposition~\ref{crit}. So we are done by Lemma~\ref{HKtheorem}.
\end{proof}

\end{document}